\documentclass[11pt, twocolumn]{article}
\usepackage[T1]{fontenc}
\usepackage[utf8]{inputenc}

\usepackage{amsmath}
\usepackage{amssymb}
\usepackage{amsfonts}
\usepackage{amsthm}

\usepackage{microtype}
\usepackage[cm]{fullpage}

\usepackage[tt=false, type1=true]{libertine}
\usepackage[libertine]{newtxmath}

\usepackage{booktabs}

\usepackage{tikz-cd}

\usepackage{url}
\usepackage{hyperref}
\hypersetup{
    colorlinks,
    linkcolor={red!50!black},
    citecolor={blue!50!black},
    urlcolor={blue!80!black}
}

\def\N{\mathbb{N}}
\def\ZF{\mathsf{ZF}}

\let\term\textbf

\theoremstyle{plain}
\newtheorem{theorem}{Theorem}
\newtheorem{question}[theorem]{Question}
\newtheorem{definition}[theorem]{Definition}

\title{Conway and Doyle Can Divide by Three, But I Can't}
\author{Patrick Lutz\footnote{University of California, Los Angeles, Department of Mathematics Email:\ pglutz@math.ucla.edu}}
\date{}

\begin{document}

\maketitle

\begin{abstract}
Conway and Doyle have claimed to be able to divide by three. We attempt to replicate their achievement and fail. In the process, we get tangled up in some shoes and socks and forget how to multiply.
\end{abstract}

\section{Introduction}

In the paper ``Division by Three'' \cite{doyle1994division}, Conway and Doyle show that it is possible to divide by 3 in cardinal arithmetic, even without the axiom of choice. Actually, they show that it is possible to divide by $n$ for all natural numbers $n > 0$; they called their paper ``Division by Three'' rather than ``Division by $n$'' because the case $n = 3$ seems to capture all the difficulty of the full result. More precisely, they give a proof of the following theorem.

\begin{theorem}
\label{thm-shoe_division}
It is provable in $\ZF$ (Zermelo-Fraenkel set theory without the axiom of choice) that for any natural number $n > 0$ and any sets $A$ and $B$, if $|A\times n| = |B \times n|$ then $|A| = |B|$.
\end{theorem}

Here we are using the notation $A\times n$ to denote the set $A\times\{1,2,\ldots, n\}$ and the notation $|A| = |B|$ to mean that there is a bijection between $A$ and $B$.

The purpose of this article is to question whether the statement of Theorem \ref{thm-shoe_division} is really the correct definition of ``dividing by $n$ without choice.'' We will propose an alternative statement, show that it is not provable without the axiom of choice, and explain what all this has to do with Bertrand Russell's socks.

Of course, none of this should be taken too seriously. I'm not really here to argue about what ``division by $n$ without choice'' means. Instead, the goal is to have fun with some interesting mathematics, and the question of what ``division by $n$ without choice'' should really mean is merely an inviting jumping-off point.

\subsection*{Mathematics Without Choice}

What does it mean to do math without the axiom of choice? In brief, it means that if we are proving something and want to describe a construction that requires infinitely many choices then we must describe explicitly how these choices are to be made, rather than just assuming that they can be made any-which-way when the time comes.

There is a well-known example, due to Bertrand Russell, that illustrates this issue. Suppose there is a millionaire who loves to buy shoes and socks. Every day, he buys a pair of shoes and a pair of socks, and after infinitely many days have passed, he has amassed infinitely many pairs of each. He then asks his butler to pick out one shoe from each pair for him to display in his foyer. The butler wants to make sure he is following the millionaire's instructions precisely, so he asks how to decide which shoe to pick from each pair. The millionaire replies that he can pick the left shoe each time. The next day, the millionaire decides he would also like to display one sock from each pair and so he asks the butler to do so. When the butler again asks how he should decide which sock to pick from each pair, the millionaire is stymied---there is no obvious way to distinguish one sock in a pair from the other.\footnote{When Russell introduced this example, he was careful to point out that in real life there actually are ways to distinguish between socks---for instance, one of them probably weighs slightly more than the other---but he asked for ``a little goodwill'' on the part of the reader in interpreting the example.}

The point of this example is that if we have a sequence $\{A_i\}_{i \in \N}$ of sets of size 2, then there is no way to prove without the axiom of choice that $\Pi_{i \in \N} A_i$ is nonempty. Doing so would require explicitly constructing an element of $\Pi_{i \in \N}A_i$, which is analogous to giving a way to choose one sock from each pair in the millionaire's collection. On the other hand, if we have a fixed ordering on each set $A_i$ in the sequence, then we \emph{can} show without choice that $\Pi_{i \in \N}A_i$ is nonempty, just as it was possible to choose one shoe from each pair.

Russell's story about the shoes and socks may seem like just a charming and straightforward illustration of the axiom of choice, but we will return to it a few times throughout this article and see that there is more to it than is initially apparent.

\section{Failing to Divide by Three}

\subsection*{You Can Divide by Three}

As we mentioned earlier, in the paper ``Division by Three,'' Conway and Doyle prove without the axiom of choice that for any natural number $n > 0$ and any sets $A$ and $B$, if $|A \times n| = |B \times n|$ then $|A| = |B|$. What this requires is giving an explicit procedure to go from a bijection between $A\times n$ and $B\times n$ to a bijection between $A$ and $B$.

This result has a long history. It was (probably) first proved by Lindenbaum and Tarski in 1926 \cite{lindenbaum1926communication}, but the proof was not published and seems to have been forgotten. The first published proof was by Tarski in 1949 and is regarded as somewhat complicated \cite{tarski1949cancellation}. Conway and Doyle gave a simpler (and more entertainingly exposited) proof, which they claimed may be the original proof by Lindenbaum and Tarski. Later, the proof was simplified even more by Doyle and Qiu in the paper ``Division by Four'' \cite{doyle2015division}. There is also a charming exposition of Doyle and Qiu's proof in the article ``Pangalactic Division'' by Schwartz \cite{schwartz2015pan}.

\subsection*{Can You Divide by Three?}

Does the statement of Theorem \ref{thm-shoe_division} really capture what it means to divide by $n$ without choice? To explain what we mean, we first need to say a little about how division by $n$ is proved. Recall that we are given a bijection between $A\times n$ and $B\times n$, and we need to construct a bijection between $A$ and $B$. We can think of both $A\times n$ and $B\times n$ as unions of collections of disjoint sets of size $n$. Namely,
\begin{align*}
A\times n &= \bigcup_{a \in A} \{(a, 1), (a, 2), \ldots, (a, n)\}\\
B\times n &= \bigcup_{b \in B} \{(b, 1), (b, 2), \ldots, (b, n)\}.
\end{align*}
A key point, which every known proof uses, is that we can simultaneously order every set in the two collections using the ordering induced by the usual ordering on $\{1, 2, \ldots, n\}$.

But if we are already working without the axiom of choice, this seems like an unfair restriction. Why not also allow collections of \emph{unordered} sets of size $n$? This gives us an alternative version of ``division by $n$ without choice'' in which we replace the collections $A\times n$ and $B\times n$ with collections of unordered sets of size $n$ (we will give a precise statement of this version below). Since collections of ordered sets of size $n$ behave like the pairs of shoes from Russell's example while collections of unordered sets of size $n$ behave like the pairs of socks, we will refer to the standard version as ``shoe division'' and the alternative version as ``sock division.''

\begin{definition}
Suppose $n > 0$ is a natural number. \term{Shoe division by} {\boldmath$n$} is the principle that for any sets $A$ and $B$, if $|A\times n| = |B \times n|$ then $|A| = |B|$.
\end{definition}

\begin{definition}
Suppose $n > 0$ is a natural number. \term{Sock division by} {\boldmath$n$} is the principle that for any sets $A$ and $B$ and any collections $\{X_a\}_{a \in A}$ and $\{Y_b\}_{b \in B}$ of disjoint sets of size $n$, if $|\bigcup_{a \in A}X_a| = |\bigcup_{b \in B}Y_b|$ then $|A| = |B|$.
\end{definition}

Since we know that shoe division by $n$ is provable without the axiom of choice, it is natural to wonder whether the same is true of sock division. 

By the way, this is not the first time that someone has asked about the necessity of having collections of ordered rather than unordered sets when dividing by $n$ in cardinal arithmetic. In the paper ``Equivariant Division'' \cite{bajpai2017equivariant}, Bajpai and Doyle consider when it is possible to go from a bijection $A\times n \to B \times n$ to a bijection $A \to B$ when the bijections are required to respect certain group actions on $A$, $B$, and $n$. Since the axiom of choice can be considered a way to ``break symmetries,'' the question of whether sock division is provable without choice is conceptually very similar to the questions addressed by Bajpai and Doyle.

\subsection*{You Can't Divide by Three}

In this section we will show that sock division by 3 is not provable without the axiom of choice. In fact, neither is sock division by 2 or, for that matter, sock division by $n$ for any $n > 1$.

\begin{theorem}
\label{thm-sock_division}
For any natural number $n > 1$, the principle of sock division by $n$ is not provable in $\ZF$.
\end{theorem}

\begin{proof}
We will show that if sock division by 2 is possible then it is also possible to choose socks for Bertrand Russell's millionaire. The full theorem follows by noting that the proof works not just for human socks but also for socks for octopi with $n > 1$ tentacles.

More precisely, suppose sock division by 2 does hold and let $\{A_i\}_{i \in \N}$ be a sequence of disjoint sets of size 2. We will show that $\Pi_{i \in \N}A_i$ is nonempty by constructing a choice function for $\{A_i\}_{i \in \N}$. We can picture $\{A_i\}_{i \in \N}$ as a sequence of pairs of socks.
\hspace*{7pt}\includegraphics[trim= 10 80 0 80, clip, width=\columnwidth]{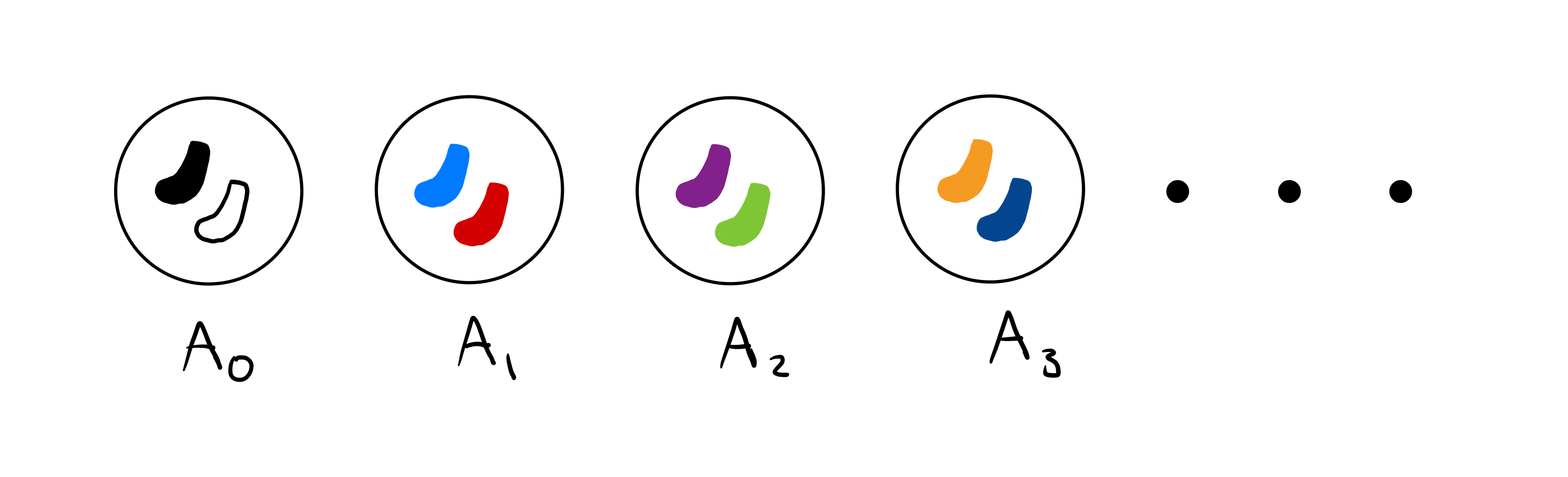}

Now consider taking a single pair of socks---say $A_0 = \{x_0, y_0\}$---and forming the Cartesian product of this pair with the set $\{0, 1\}$. This gives us a 4 element set, depicted by the grid below.
\begin{center}
\hspace*{-10pt}\includegraphics[trim= 10 80 0 20, clip, width=0.4\columnwidth]{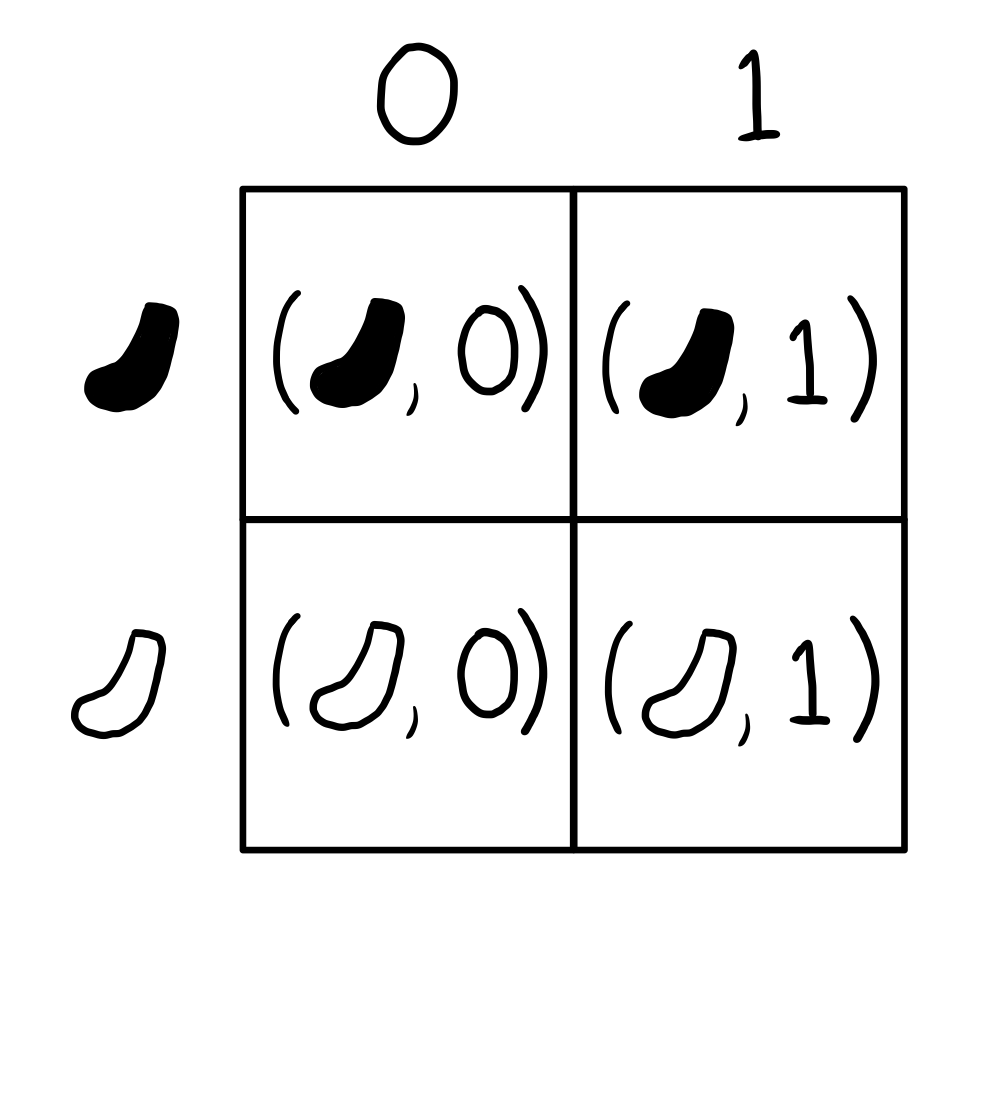}
\end{center}
We will divide this 4 element set into a pair of 2 element sets in two different ways. First, we can take the rows of the grid to get the sets $\{(x_0, 0), (x_0, 1)\}$ and $\{(y_0, 0), (y_0, 1)\}$.
\begin{center}
\includegraphics[trim= 0 80 0 80, clip, width=0.75\columnwidth]{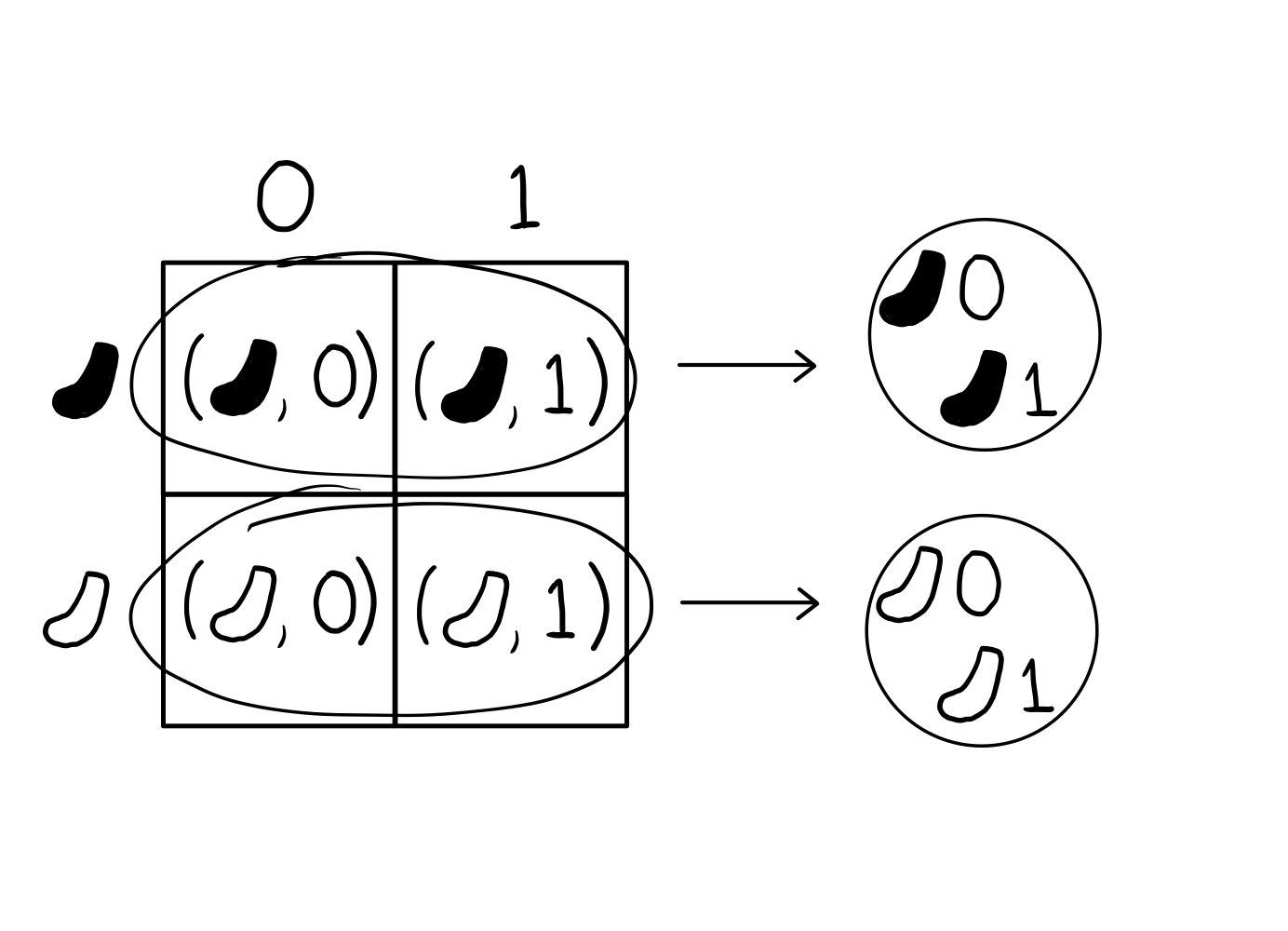}
\end{center}
Second, we can take the columns of the grid to get the sets $\{(x_0, 0), (y_0, 0)\}$ and $\{(x_0, 1), (y_0, 1)\}$.
\begin{center}
\hspace{-14pt}\includegraphics[trim= 10 40 0 60, clip, width=0.7\columnwidth]{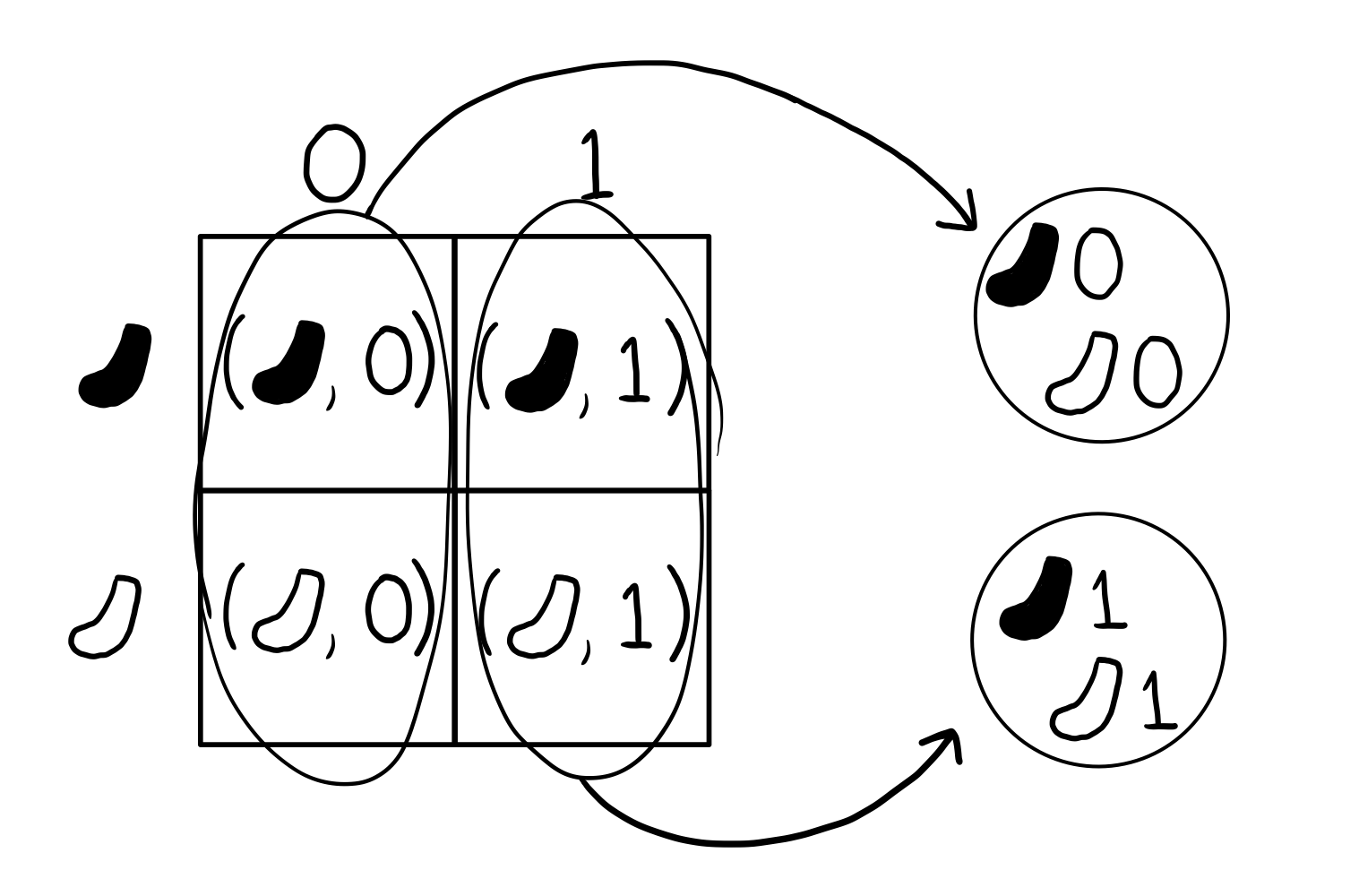}
\end{center}

If we repeat this procedure for every pair of socks, we end up with two collections of disjoint sets of size 2---one consisting of the rows of the grids formed from each pair and the other consisting of the columns.
\begin{center}
\includegraphics[trim= 0 50 0 80, clip, width=0.9\columnwidth]{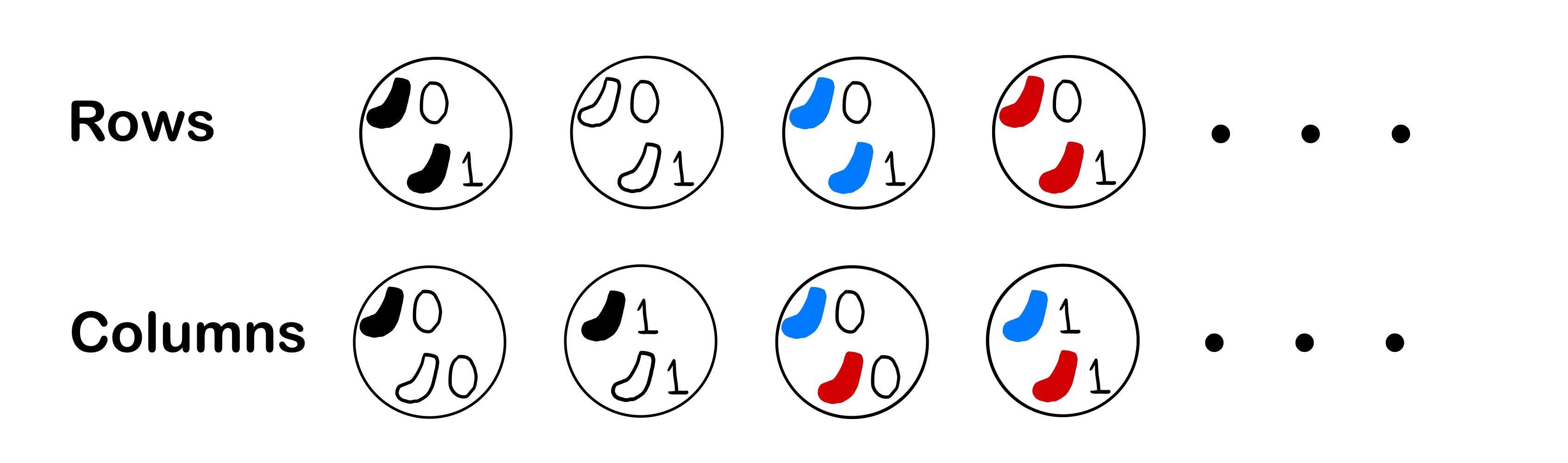}
\end{center}

Now we will observe a few things about these two collections. 
\begin{itemize}
    \item First, each set in the collection of rows has the form $\{(x, 0), (x, 1)\}$ for some $x \in \bigcup_{i \in \N}A_i$, so we can think of the collection of rows as being indexed by $\bigcup_{i \in \N}A_i$ (i.e.\ indexed by the individual socks).
    \item Second, each set in the collection of columns either has the form $A_i\times\{0\}$ for some $i \in \N$ or the form $A_i \times \{1\}$ for some $i \in \N$, so we can think of the collection of columns as being indexed by $\N\times \{0, 1\}$.
    \item Lastly, the union of the collection of rows and the union of the collection of columns are identical---they are both just equal to $\bigcup_{i \in \N}(A_i\times\{0, 1\})$.
\end{itemize}
\begin{center}
\hspace*{10pt}\includegraphics[trim= 0 170 0 80, clip, width=0.9\columnwidth]{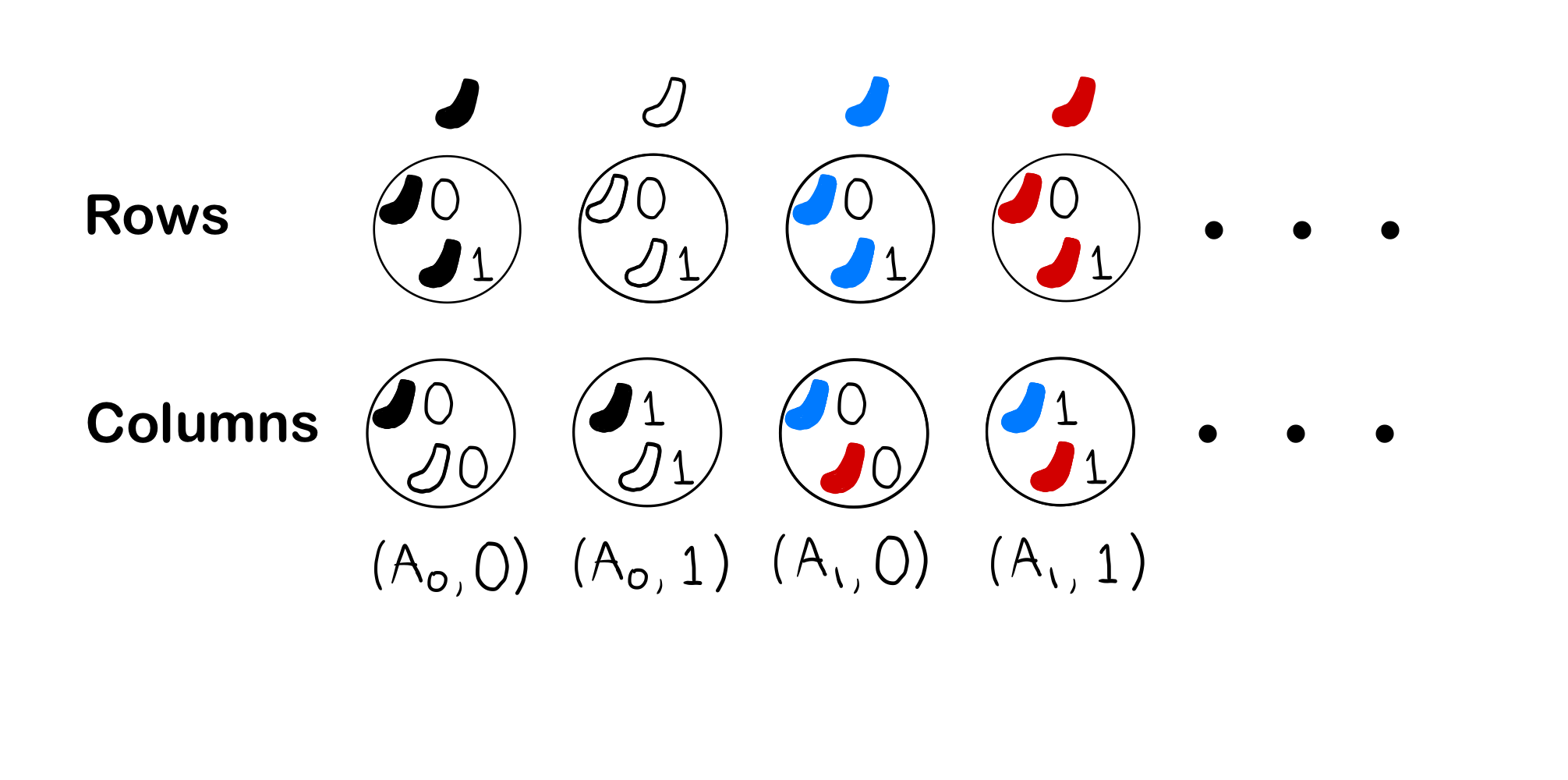}
\end{center}

The principle of sock division by 2 says that if the unions of two collections of disjoint sets of size 2 are in bijection then the sets indexing those collections are also in bijection. Thus we can conclude that there is a bijection $f \colon (\bigcup_{i \in \N}A_i) \to \N\times\{0,1\}$.

We can now describe how to choose one sock from each pair. Consider a pair of socks, $A_i = \{x, y\}$. We know that $x$ is mapped by $f$ to some pair $(n, b) \in \N\times \{0, 1\}$ and $y$ is mapped to some other pair, $(m, c)$. We can choose between $x$ and $y$ by picking whichever one is mapped to the smaller pair in the lexicographic ordering on $\N\times \{0, 1\}$. 
\end{proof}

\section{Cardinal Arithmetic and the Power of Sock Division}

In this section we will discover another view of sock division by considering how to define multiplication of cardinals without choice.

\subsection*{Shoes and Socks, Revisited}

Suppose we have two sets, $A$ and $B$. How should we define the product of their cardinalities? The standard answer is that it is the cardinality of their Cartesian product---i.e.\ $|A\times B|$. But there is another possible definition. Suppose $\{X_a\}_{a \in A}$ is a collection of disjoint sets such that each $X_a$ has the same cardinality as $B$. Since taking a disjoint union of sets corresponds to taking the sum of their cardinalities, we can think of $|\bigcup_{a \in A} X_a| = \sum_{a \in A}|X_a|$ as an alternative definition of ``the cardinality of $A$ times the cardinality of $B$.'' 

One way to think of these two definitions is that the first interprets multiplication as the area of a rectangle, while the second interprets it as repeated addition. 

\begin{center}
\begin{tabular}{@{}lc@{}}
  \toprule
  \textbf{Multiplication is \ldots} & \\
  \midrule
  The area of a rectangle: & $|A|\times |B| = |A\times B|$\\
  Repeated addition: & $|A|\times|B| = |\bigcup_{a \in A}X_a|$\\
\bottomrule
\end{tabular}
\end{center}

One problem with thinking of multiplication as repeated addition, however, is that without the axiom of choice, it may not be well-defined. In particular, it is possible to have two collections $\{X_a\}_{a \in A}$ and $\{Y_a\}_{a \in A}$ of disjoint sets of size $|B|$ such that $|\bigcup_{a \in A}X_a| \neq |\bigcup_{a \in A} Y_a|$. In fact, this is actually the original context for Russell's example about shoes and socks. The following passage is from his 1919 book \emph{Introduction to Mathematical Philosophy} \cite{russell1919introduction} (note that he refers to the axiom of choice as the ``multiplicative axiom,'' since it guarantees that every nonzero product of nonzero cardinalities is nonzero).

\begin{quote}
Another illustration may help to make the point clearer. We know that $2\times \aleph_0 = \aleph_0$. Hence we might suppose  that  the  sum  of $\aleph_0$ pairs  must  have $\aleph_0$ terms. But this, though we can prove that it is sometimes the case, cannot be proved to happen always unless we assume the multiplicative axiom. This is illustrated by the millionaire who bought a pair of socks whenever he bought a pair of boots, and never at any other time, and who had such a passion for buying both that at last he had $\aleph_0$ pairs of boots and $\aleph_0$ pairs of socks. The problem is: How many boots had he, and how many socks? One would naturally suppose that he had twice as many boots and twice as many socks as he had pairs of each, and that therefore he had $\aleph_0$ of each, since that number is not increased by doubling. But this is an instance of the difficulty, already noted, of connecting the sum of $\nu$ classes each having $\mu$ terms with $\mu\times \nu$. Sometimes this can be done, sometimes it cannot. In our case it can be done with the boots, but not with the socks, except by some very artificial device.
\end{quote}

\subsection*{Multiplication vs.\ Division}

Let's revisit the difference between shoe division and sock division in light of what we have just discussed. When discussing ``division by $n$ without choice,'' we have implicitly defined division in terms of multiplication. Being able to divide by $n$ means that whenever we have $|A|\times n = |B| \times n$, we can cancel the $n$'s to get $|A| = |B|$. The only difference between shoe division and sock division is what definition of multiplication is used (i.e.\ what is meant by $|A|\times n$ and $|B|\times n$). In shoe division, multiplication is interpreted in the usual way, i.e.\ as ``the area of a rectangle.'' In sock division, it is interpreted as ``repeated addition.''

When we view shoe division and sock division in this way, it is clear that if ``multiplication by $n$ as repeated addition of $n$'' is well-defined then sock division by $n$ holds (because in this case it is equivalent to shoe division by $n$). Thus it is natural to ask what the exact relationship is between these two principles.

A priori, they are fairly different statements. Sock division by $n$ says that if we have two collections $\{X_a\}_{a \in A}$ and $\{Y_b\}_{b \in B}$ of disjoint sets of size $n$ then we can go from a bijection between $\bigcup_{a \in A}X_a$ and $\bigcup_{b \in B}Y_b$ to a bijection between $A$ and $B$ while ``multiplication by $n$ as repeated addition of $n$ is well-defined'' says that we can go from a bijection between $A$ and $B$ to a bijection between $\bigcup_{a \in A}X_a$ and $\bigcup_{b \in B}Y_b$. However, it turns out that the two principles are actually equivalent and the proof of this is implicit in our proof of Theorem \ref{thm-sock_division}. Let's make all of this more precise.

\begin{definition}
Suppose $n > 0$ is a natural number. \term{Multiplication by {\boldmath$n$} is equal to repeated addition of {\boldmath$n$}} is the principle that for any set $A$ and any collection $\{X_a\}_{a \in A}$ of disjoint sets of size $n$, $|\bigcup_{a \in A}X_a| = |A\times n|$.
\end{definition}

What we can show is that in $\ZF$, the principle of sock division by $n$ is equivalent to the principle that multiplication by $n$ is equal to repeated addition of $n$. 

\begin{theorem}
\label{thm-sock_multiplication}
It is provable in $\ZF$ that for any natural number $n > 0$, the principle of sock division by $n$ holds if and only if the principle that multiplication by $n$ is equal to repeated addition by $n$ holds.
\end{theorem}

\begin{proof}
($\impliedby$) First suppose ``multiplication by $n$ is equal to repeated addition of $n$'' holds. Let $A$ and $B$ be any sets and let $\{X_a\}_{a \in A}$ and $\{Y_b\}_{b \in B}$ be two collections of disjoint sets of size $n$ such that $|\bigcup_{a \in A}X_a| = |\bigcup_{b \in B}Y_b|$. Applying ``multiplication is repeated addition,'' we have
\begin{align*}
|A\times n| = \left|\bigcup\nolimits_{a \in A}X_a\right| = \left|\bigcup\nolimits_{b \in B}Y_b\right| = |B\times n|.
\end{align*}
And by applying shoe division by $n$, we get $|A| = |B|$.

\medskip
\noindent
($\implies$) Now suppose sock division by $n$ holds and let $A$ be any set and $\{X_a\}_{a \in A}$ be a collection of disjoint sets of size $n$. Consider the set $\bigcup_{a \in A}(X_a \times n)$. We can view this set as a union of a collection of disjoint sets of size $n$ in two different ways:
\begin{align*}
\bigcup\nolimits_{a \in A}(X_a \times n) &= \bigcup\nolimits_{a \in A,\ i \le n} \{(x, i) \mid x \in X_a\}\\
\bigcup\nolimits_{a \in A}(X_a \times n) &= \bigcup\nolimits_{a \in A,\ x \in X_a} \{(x, i) \mid i \le n\}.
\end{align*}
The first of these two collections is indexed by $A\times n$ and the second is indexed by $\bigcup_{a \in A}X_a$. And since the unions of the two collections are identical, sock division implies that $|A\times n| = |\bigcup_{a \in A}X_a|$.
\end{proof}

\subsection*{Sock Geometry}

Just how powerful is sock division? We have just seen that if sock division by $n$ holds then multiplication by $n$ is equal to repeated addition of $n$. In other words, for any set $A$ and any collection $\{X_a\}_{a \in A}$ of disjoint sets of size $n$, there is a bijection between $\bigcup_{a \in A}X_a$ and $A\times n$. However, this bijection does not necessarily respect the structure of $\bigcup_{a \in A}X_a$ and $A\times n$ as collections of size $n$ sets indexed by $A$: we are not guaranteed that the image of each $X_a$ is equal to $\{a\}\times n$. It seems reasonable, then, to ask whether sock division by $n$ implies the existence of a bijection that does respect this structure.

It is natural to phrase this question using terms from geometry, and in particular in the language of fiber bundles. It is possible to understand everything in this section even if you do not know what a bundle is, but our choice of terminology may seem a bit odd.

We can think of a collection $\{X_a\}_{a \in A}$ of disjoint sets of size $n$ as a kind of bundle over $A$. We will refer to it as an \term{{\boldmath$n$}-sock bundle}, or just a \term{sock bundle} for short. We can think of the index set $A$ as the \term{base space} of the bundle and the union $\bigcup_{a \in A}X_a$ as the \term{total space}. If $\{X_a\}_{a \in A}$ and $\{Y_a\}_{a \in A}$ are two $n$-sock bundles then a \term{sock bundle isomorphism} between them is a bijection $f \colon \bigcup_{a \in A}X_a \to \bigcup_{a \in A}Y_a$ such that for each $a$, the image of $X_a$ is $Y_a$---in other words, such that the following diagram commutes (where $\pi$ and $\pi'$ denote the natural projections $\bigcup_{a \in A}X_a \to A$ and $\bigcup_{a \in A}Y_a \to A$).
\begin{center}
\begin{tikzcd}
\bigcup_{a \in A}X_a \arrow[r, "f"] \arrow[d, "\pi"] & \bigcup_{a \in A}Y_a \arrow[d, "\pi'"] \\
A \arrow[r, "\text{id}"]                             & A                                     
\end{tikzcd}
\end{center}
We will refer to $A\times n$ as the \term{trivial {\boldmath$n$}-sock bundle}\footnote{It would also be reasonable to call it the $n$-shoe bundle.} and call a sock bundle \term{trivializable} if it is isomorphic to $A\times n$. We summarize some of these terms in the table below.

\begin{center}
\begin{tabular}{@{}lc@{}}
  \toprule
  \textbf{Geometry} & \textbf{Sock Geometry}\\
  \midrule
    Bundle & $\{X_a\}_{a \in A}$\\
    Total space & $\bigcup_{a \in A}X_a$\\
    Base space & $A$\\
    Trivial bundle & $A\times n$\\
    Trivializable bundle & $\bigcup_{a \in A}X_a \cong A\times n$\\
\bottomrule
\end{tabular}
\end{center}

Restated in these terms, here's the question we asked above.

\begin{question}
\label{q:sock_geometry}
Let $n > 0$ be a natural number. Does $\ZF$ prove that sock division by $n$ implies that all $n$-sock bundles are trivializable?
\end{question}

There is at least one special case in which this question has a positive answer: when the base space $A$ can be linearly ordered. To see why, suppose sock division by $n$ holds, let $A$ be any set and let $\preceq$ be a linear order on $A$. If $\{X_a\}_{a \in A}$ is a collection of disjoint sets of size $n$ then we know by Theorem \ref{thm-sock_multiplication} that sock division by $n$ implies that there is a bijection $f \colon \bigcup_{a \in A}X_a \to A\times n$. Since $\preceq$ is a linear order on $A$, we can linearly order $A\times n$ using $\preceq$ and the standard ordering on $\{1, 2, \ldots, n\}$. Thus we can use $f$ to simultaneously linearly order all the $X_a$'s and thereby trivialize the bundle $\{X_a\}_{a \in A}$.

\subsection*{Sock Division and Divisibility}

We will end with one more question about the power of sock division. Consider the question of what it means for the cardinality of a set $A$ to be divisible by a natural number $n$. It seems natural to define divisibility in terms of multiplication: $|A|$ is divisible by $n$ if there is some set $B$ such that $|A| = |B|\times n$. However, we saw above that there are two possible ways to interpret $|B|\times n$, and without the axiom of choice these two are not necessarily equivalent. Thus we have two possible notions of divisibility by $n$ without choice, one based on interpreting multiplication as the area of a rectangle and the other based on interpreting multiplication as repeated addition. 

These two notions were studied by Blair, Blass and Howard~\cite{blair2005divisibility}, who called them \term{strong divisibility by \boldmath$n$} and \term{divisibility by \boldmath$n$}, respectively. To help distinguish the two notions, we will refer to divisibility by $n$ as \term{weak divisibility by \boldmath$n$}.

\begin{definition}
Suppose $n > 0$ is a natural number. A set $A$ is \term{strongly divisible by $n$} if there is a set $B$ such that $|A| = |B\times n|$ and \term{weakly divisible by $n$} if there is a set $B$ and a collection $\{X_b\}_{b \in B}$ of disjoint sets of size $n$ such that $|A| = |\bigcup_{b \in B}X_b|$.
\end{definition}

In the language of sock geometry, one might say that $A$ is strongly divisible by $n$ if it is in bijection with the total space of some trivial $n$-sock bundle and weakly divisible by $n$ if it is in bijection with the total space of any $n$-sock bundle.

It is clear that strong divisibility by $n$ implies weak divisibility by $n$, but without the axiom of choice, weak divisibility does not imply strong divisibility (for example, this is implied by the results of Herrlich and Tachtsis in~\cite{herrlich2006number}). Since sock division by $n$ implies that multiplication by $n$ is equal to repeated addition of $n$, sock division by $n$ implies that strong and weak divisibility by $n$ are equivalent. However, it is not clear whether the converse holds and this seems like an interesting question.

\begin{question}
\label{q:divisibility}
Let $n > 0$ be a natural number. Does $\ZF$ prove that if strong and weak divisibility by $n$ are equivalent then sock division by $n$ holds?
\end{question}

I believe that answering questions~\ref{q:sock_geometry} and~\ref{q:divisibility} above would give insight into the world of set theory without choice more generally.

\section{Acknowledgements}

Thanks to Peter Doyle for a lively email conversation on the topic of this paper and for coming up with the name ``sock division,'' an anonymous reviewer for suggesting question~\ref{q:divisibility}, Rahul Dalal for reading and commenting on a draft, Brandon Walker for inadvertently providing me with the motivation to work on this topic and, of course, Conway, Doyle, Qiu and all the rest for inspiration.

\bibliographystyle{alpha}
\bibliography{bibliography}

\end{document}